\documentclass[a4paper,11pt,reqno,noindent]{amsart}
\usepackage[centertags]{amsmath} 
\usepackage{amsfonts,amssymb,amsthm}
\usepackage{mathtools}
\usepackage{mathrsfs}
\usepackage{xcolor}
\usepackage{graphicx}

\usepackage[english]{babel} \usepackage{newlfont} \usepackage{color}
\usepackage[body={15cm,21.5cm},centering]{geometry} \usepackage{fancyhdr}
 \usepackage{esint} 

\usepackage[active]{srcltx}

\newtheorem{theorem}{Theorem} 
 
\newtheorem{corollary}{Corollary} 
\newtheorem{definition}{Definition}

\theoremstyle{definition} 
\newtheorem{remark}{Remark}

 \newcommand{\diag}{\mathrm{diag}}

\newcommand{\e}{\varepsilon}

\newcommand{\R}{\mathbb{R}}

 \newcommand{\id}{\mathrm{Id}}
\newcommand{\n}[1]{\mathbf{n}_{[#1]}}

\newcommand{\C}{\mathcal C}

\title{Godunov variables and convex entropy for relativistic fluid dynamics with bulk viscosity}
\author[H.~Olbermann]{Heiner Olbermann} 
\address[Heiner Olbermann]{Institut de Recherche en Math\'ematique et Physique, UCLouvain, 1348 Louvain-la-Neuve, Belgium}
\email[Heiner Olbermann]{heiner.olbermann@uclouvain.be}

\date{\today} 

\begin{document}

\maketitle

\begin{abstract} Based on the conservation-dissipation formalism proposed by Zhu and collaborators \cite{zhu2015conservation} we formulate a general version of the Israel-Stewart theory for relativistic fluid dynamics with bulk viscosity. Our generalization consists in allowing for a wide range of dependence of the  entropy density   on the bulk viscosity. We show the existence of   Godunov-Boillat variables for this model. By known properties of systems possessing such  variables,  this provides an alternative proof of the recently established  existence of solutions for the Israel-Stewart theory locally in time, and a proof that entropy production is positive across  weak Lax shocks. 
  \end{abstract}

\section{Introduction}
The theories of relativistic fluid dynamics including viscosity by M\"uller \cite{muller1967paradoxon} and Israel \& Stewart \cite{israel1979transient} are based on the idea to introduce thermodynamic fluxes as independent dynamic variables. For each of the new variables, one equation of motion has to be added to the conservation laws of energy-momentum and particle density. There are several variants of the theory, depending on the number of additional variables one includes, and other modeling choices; we will call all of these different theories M\"uller-Stewart-Israel (MIS) type theories. They have a long and successful history as a tool in simulations, in particular for quark-gluon plasmas \cite{romatschke2019relativistic}. The mathematical underpinning of MIS theories   is  less developed. Only recently  statements about well-posedness  have been proved: For the theory including  bulk viscosity the existence of solutions in Sobolev spaces has been shown \cite{bemfica2019causality}\footnote{In \cite{bemfica2019causality}, the equations of motion for the fluid are coupled with the Einstein equations, and existence is shown for this larger system. We will limit ourselves here to a fixed Minkowskian spacetime.}, and for the theory including bulk and shear viscosity as well as heat fluxes in Gevrey spaces \cite{bemfica2021nonlinear,MR3927412}.  The proof in \cite{bemfica2019causality} is based on writing the equations as a first order symmetric hyperbolic system, for which the existence of solutions locally in time is guaranteed by classical results by Friedrichs \cite{friedrichs1954symmetric} and Lax \cite{lax1957hyperbolic} (see also the textbooks \cite{dafermos2005hyperbolic,MR2744149}). In \cite{disconzi2020breakdown} it has been shown that (at least for certain initial conditions and choices of the parameter of the model) the MIS equations with bulk viscosity produce breakdown of smooth solutions in finite time.
\medskip

In the present article, we will formulate a MIS-type theory with bulk viscosity following the so-called conformal-dissipation formalism (CDF) by Zhu, Hong, Yang \& Yong \cite{zhu2015conservation}. The CDF is based on elements from extended irreversible thermodynamics \cite{jou1996extended}, rational extended thermodynamics \cite{muller2013rational} and the GENERIC (general equations for non-equilibrium reversible irreversible coupling) formalism \cite{ottinger2005beyond}.  It consists in enlarging the state space of thermodynamic variables by initially non-specified dynamic non-equilibrium variables. The formalism is set up in a way to guarantee that the equations of motions are a hyperbolic relaxation system, i.e., of the form
\[
  \partial_\alpha F^\alpha(U)=Q(U)\,,
\]
where $U$ is an  unknown $m$-dimensional vector valued function, and $Q,F^\alpha$ are given $m$-dimensional vector  valued functions, and we are using the convention that double greek indices are being summed over ($\partial_\alpha F^\alpha(U)\equiv \sum_{\alpha=0}^d \frac{\partial}{\partial{x^\alpha}}F^\alpha(U)$). In particular, this form allows for a straightforward definition of a concept of weak solutions, including discontinous shocks  to which Lax's classic analysis \cite{lax1973hyperbolic} applies. Dissipative relativistic fluid equations  in ``divergence form'' have been considered in \cite{geroch1991causal,geroch1990dissipative}, but to the best of the author's knowledge this  is  the first time that it is pointed out  that the MIS equations with bulk viscosity can be cast in  this form.

\medskip

To set up our model, we will introduce only one dynamic non-equilibrium variable $\C$ in addition to the dynamic variables present in the three-dimensionsl relativistic Euler equations. This variable will be associated with the bulk viscosity $\pi$. The specific entropy density $s$ will then depend on the equilibrium variables (on the internal energy per particle $\e$ and the specific volume per particle  $\nu$, say) and the additionally introduced  variable $\C$. The relation of $\C$ to the bulk viscosity $\pi$ will be fixed through Legendre-Fenchel duality, and the equation of motion for $\C$ will be fixed by the requirement that the entropy production shall be non-negative. The standard MIS equations for bulk viscosity become a special case of our setting, defined by $s(\e,\nu,\C)=s^{\mathrm{eq}}(\e,\nu)-\frac{\C^2}{2\theta^{\mathrm{eq}}(\e,\nu)}$, where $s^{\mathrm{eq}},\theta^{\mathrm{eq}}$ denote the equilibrium entropy density per particle and temperature, respectively. 

\medskip

We then prove that the generalized MIS equations with bulk viscosity possess the structure of a \emph{causal covariant Godunov-Boillat system}. This is the relativistic version of the structure identified by Godunov \cite{godunov1961interesting} that provides natural field variables that render  the equations symmetric hyperbolic. In the covariant setting, the theory has been developed by  Ruggeri \& Strumia \cite{ruggeri1981main}. The proof that our model is a Godunov-Boillat system will follow the same lines as in that reference. 
It is also proved there that entropy production is positive across weak Lax shocks in Godunov-Boillat systems. By our  theorem, this applies in particular  the  generalized MIS model with bulk viscosity, a fact that we note as a corollary. Godunov-Boillat variables have recently  been used in the construction of dissipative relativistic fluid dynamics in \cite{freistuhler2020class,freistuhler2017causal,freistuhler2014causal,freistuhler2018causal}.  

\medskip

The rest of this paper is structured as follows: 
In Section \ref{sec:constr-isra-stew}, we will construct our model based on the CDF. In Section \ref{sec:godun-boill-vari} we show that this model is a causal covariant Godunov-Boillat system (Theorem \ref{thm:main}), and state the corollary on entropy production across weak Lax shocks.

\bigskip

\subsection*{Notation}
We will use throughout the Minkowskian metric  $g=\diag(-1,1,1,1)$ on $\R^4$. Greek indices $\alpha,\beta,\dots$ will denote space-time components, and we will sum over repeated indices. The inverse of $g_{\alpha\beta}$ is denoted by $g^{\alpha\beta}$. Indices can be lifted or lowered  by $g^{\alpha\beta}$ and $g_{\alpha\beta}$ respectively. On $\R^4$ we will consider vector fields with values in $\R^6$ (i.e., we consider a relativistic 6-field theory). Indices of the components of these vectors will be Latin letters, $i,j,k,\dots\in \{0,\dots,5\}$, and summations will  be indicated for these. In some equations, we will only consider field  field components carrying indices $0,\dots,3$ - in these cases, we allow for greek indices, and summation will be indicated.

\section{Construction of Israel-Stewart theory with bulk viscosity via the conservation-dissipation formalism}
\label{sec:constr-isra-stew}

Consider a relativistic viscous fluid with stress energy tensor $T^{\alpha\beta}$ and four-current $J^\alpha$ given by 
\[
  \begin{split}
  T^{\alpha\beta}&= n \e u^\alpha u^\beta+(p+\pi)\Delta^{\alpha\beta}\\
  J^\alpha&= n  u^\alpha
\end{split}
\]
where  $ n =\nu^{-1}$ is the particle density, $\e$ the  internal energy per particle, $u$ the hydrodynamic velocity satisfying $u^\alpha u_\alpha=-1$,  $\Delta^{\alpha\beta}$ the symmetric spatial projection $\Delta^{\alpha\beta}=g^{\alpha\beta}+u^\alpha u^\beta$, $p$ the equilibrium pressure
, $\pi$ the bulk viscous pressure. We will also use the notation $\vec u=(u^1,u^2,u^3)^T$.

\medskip

Now we  adapt the   conservation-dissipation formalism  by Zhu, Hong, Yang, Yong \cite{zhu2015conservation} to the relativistic setting. 
One introduces a new field variable $\C$, that is supposed to be conjugate to the bulk viscosity $\pi$ (in the sense of \eqref{eq:7} below) and supposes that the equation of state is given in the form
\[
  s=s(\e,\nu,\C)\,,
\]
supposing that $s$ is a convex function of its arguments.
The variable $\C$ is conjugate to $\pi$ in the sense that
  \begin{equation}\label{eq:7}
  \frac{\pi}{\theta}=-\partial_\C s\,.
\end{equation}
As usual, the relation of $\theta,p$ to the other variables is fixed by requiring  $\theta^{-1}=\partial_\e s$, $\theta^{-1}p=\partial_\nu s$.

\medskip

Let $\Omega\subset\R^6$ denote the ``field manifold'', which for definiteness we write as
\[
\Omega=\left\{U=(\vec u^T,\e,\nu,\C)^T\in\R^3\times (0,\infty)^3\right\}\,.
\]
From now on, $u^\beta$, $\beta=0,\dots,3$, as well as $\e,n,\pi,p,\theta,s,\nu$ will be considered as real-valued smooth functions on $\Omega$ (that will be partly determined by the equation of state relating the different thermodynamic quantities). Whenever a 6-tuple of functions $W=(w_0,\dots,w_5)^T:\Omega\to \R^6$ satisfies $\det D_U w(U)\neq 0$ throughout $\Omega$, we will say that $w$ provides a set of coordinates (on $\Omega$). In such a case,
we define partial  derivatives with respect to $w_j$ in the obvious way by  $\partial_{U_j}f=\sum_{i=0}^5\partial_{U_j}w_i\partial_{w_i} f $.

\bigskip



\medskip

Our aim is to complete the set of equations $\partial_\alpha T^{\alpha\beta}=\partial_\alpha J^\alpha=0$ with an equation of motion for $\C$ such that the resulting system becomes  symmetric hyperbolic. For the  four-dimensional entropy flux 
  \[
    S^\alpha=n u^\alpha s(\e,\nu,\C)\,,\quad \alpha=0,\dots,3
  \]
we require that $\nabla_\alpha S^\alpha\geq 0$, which is the relativistic formulation of the second law of thermodynamics; this requirement will  restrict the choice of  evolution equation for $\C$.  We have that
\[
  \begin{split}
  \partial_\alpha (n u^\alpha s(\nu,\e,\C))&= \frac{\partial s}{\partial \e} \partial_\alpha (n u^\alpha\e)+\frac{\partial s}{\partial \nu}\partial_\alpha (n u^\alpha \nu)+\frac{\partial s}{\partial \C}\partial_\alpha (n u^\alpha \C)\\
  &=-\theta^{-1} u_\beta \partial_\alpha(n\e u^\alpha u^\beta)+p\theta^{-1}\partial_\alpha u^\alpha-n u^\alpha\theta^{-1}\pi\partial_\alpha \C\\
  &=\theta^{-1}u_\beta\partial_\alpha\left((p+\pi)\Delta^{\alpha\beta}\right)
     +p\theta^{-1}\partial_\alpha u^\alpha-n u^\alpha\theta^{-1}\pi\partial_\alpha \C\\
  &=-\theta^{-1}\pi\partial_\alpha u^\alpha
-n u^\alpha\theta^{-1}\pi\partial_\alpha \C\,.
\end{split}
\]
Here we have used identities like $u_\beta\partial_\alpha\Delta^{\alpha\beta}=-\partial_\alpha u^\alpha$,  etc.
    
    Grouping together the terms in $\pi$ and requiring $\partial_\alpha S^\alpha\geq 0$ leads to imposing the equation of motion for $\C$, 
    \[
                    n u^\alpha\partial_\alpha \C+\partial_\alpha u^\alpha=
           -M (U) \pi\,,
        \]
  where $M$ is  strictly positive, but otherwise arbitrary.

      \bigskip

      Summarizing, the generalized MIS equations with bulk viscosity are
        \begin{equation}\label{eq:12}
        \partial_\alpha F^\alpha(U)=Q(U) 
      \end{equation}
      with 
\[
F^\alpha(U)=\left(
  \begin{array}{c}
     n\e  u^\alpha u^\beta+(p+\pi)\Delta^{\alpha\beta}\\n\, u^\alpha \\ (n\C+1)u^\alpha
  \end{array}
\right)\,,\qquad Q(U)=\left(
  \begin{array}{c}
     0\\0 \\ -M (U) \pi
  \end{array}
\right)\,.
\]

\begin{remark}
A treatment of the non-relativistic analog of this theory can be found in \cite{MR3497743,MR3379901}.
\end{remark}

\section{Godunov-Boillat variables}
\label{sec:godun-boill-vari}
      
Following \cite{ruggeri1981main,freistuhler2019relativistic}, we will show that  that the generalized MIS equations with bulk viscosity are a convex density system:

\begin{definition}
  A set of PDE
    \begin{equation}\label{eq:2}
    \partial_{\alpha} F^{\alpha}_i=Q_i\qquad i=0,\dots,n-1
  \end{equation}
  is called a causal covariant Godunov-Boillat system and $\psi=(\psi_0,\dots,\psi_{n-1})^T$ Godunov variables or main field, as well as $X^\alpha=X^\alpha(\psi)$ a potential for \eqref{eq:2} if
    \begin{equation}\label{eq:3}
    F^{\alpha }_i=\partial_{\psi_i} X^\alpha(\psi)
  \end{equation}
  and
    \begin{equation}
   \left( \frac{\partial^2 X^\beta(\psi)}{\partial \psi_j\partial\psi_i}\right)_{i,j=0,\dots,n-1} \xi_\beta \quad\text{ is  definite for each fixed vector $\xi_\beta$ with $\xi_\beta\xi^\beta<0$.} \label{eq:4}
      \end{equation}

\end{definition}

\begin{remark}
  Any causal covariant Godunov-Boillat system can be written as a first order symmetric hyperbolic system,
  \[
    \partial_\alpha F^{\alpha }=(D_\psi^2  X^\alpha) \partial_\alpha \psi=\bar Q\,.
  \]
  The properties \eqref{eq:3} and \eqref{eq:4} supply an alternative proof of the symmetrizability of the (generalized) MIS equations with bulk viscosity that has been stated in \cite{bemfica2019causality}.  Symmetrizability in turn implies  existence of solutions in Sobolev spaces locally in time (see \cite{lax1957hyperbolic,friedrichs1954symmetric,dafermos2005hyperbolic,MR2744149}).
\end{remark}


\bigskip

      Now following the convex covariant density formalism from \cite{ruggeri1981main}, the Godunov variables $\psi$ are obtained as a solution of
      \[
        \psi^T D_U F^\alpha= D_U S^\alpha\,, \qquad \alpha=0,\dots,3\,.
      \]
      Obviously this system is overdetermined. We will now exploit the necessary conditions above and will see at the end of these calculations that there exists indeed a unique solution.

      \medskip

It is convenient to introduce yet another set of field coordinates:    Let $v^\alpha=n u^\alpha$, and $V=(v^0,\vec v,\e,\C)^T$, where $\vec v=(v^1,v^2,v^3)$. The vector $F^\alpha$ comprising energy-momentum, particle density flux and the expression whose differentiation yields the additional equation of motion reads in the $V$-variables as
\[
F^\alpha=\left(
  \begin{array}{c}
     \e \nu v^\alpha v^\beta+(p+\pi)\Delta^{\alpha\beta}\\v^\alpha \\ (\C+\nu)v^\alpha
  \end{array}
\right)
\]
This yields
    \[
      D_V F^\alpha=\left( 
        \begin{array}{ccc}
                     (A^{\alpha\beta}_\mu)_{\beta,\mu=0,\dots,3} &(B^{\alpha\beta})_{\beta=0,\dots,3} & \partial_\C(p+\pi)(\Delta^{\alpha\beta})_{\beta=0,\dots,3}\\ (\delta^\alpha_\mu)_{\mu=0,\dots,3} &0 & 0\\
           (\delta_\mu^\alpha (\C+\nu)+\nu^3v^\alpha v_\mu)_{\mu=0,\dots,3} &0 &  v^\alpha
        \end{array}\right)
    \]
    with
\[
      \begin{split}
      A^{\alpha\beta}_{\mu}&=
      \frac{\partial}{\partial v^{\mu}}\left((\nu \e+\nu^2 (p+\pi))v^\alpha v^\beta+(p+\pi) g^{\alpha\beta}\right)\\
      &=\left((\e +2\nu(p+\pi))v^\alpha v^\beta+\partial_\nu(p+\pi) \Delta^{\alpha\beta}\right)\nu^3v_\mu+\left(\e\nu +\nu^2(p+\pi)\right)(\delta^\alpha_\mu v^\beta+\delta^\beta_\mu v^\alpha)
      \          \end{split}
  \]
    
   and
    
  \[
      \begin{split}
      B^{\alpha\beta}&=\partial_\e\left((\nu \e+\nu^2 (p+\pi))v^\alpha v^\beta+(p+\pi) g^{\alpha\beta}\right)\\
     &= \nu v^\alpha v^\beta+ \partial_\e\left( p+ \pi\right)\Delta^{\alpha\beta}\,.
 \end{split}
   \]
Furthermore we have
\[
        D_V S^\alpha=\left(
            (\delta^{\alpha}_\mu s+\nu^3 p\theta^{-1}v^\alpha v_\mu)_{\mu=0,\dots,3}\,\,,\,\,\theta^{-1} v^\alpha\,\,,\,\, -v^\alpha\theta^{-1}\pi\right)
      \]

      We insert these formulae in $D_VS^\alpha=\psi^TD_VF^\alpha$, and have
      \[
        \begin{split}
        B^{\alpha\beta}\psi_\beta=\theta^{-1}v^\alpha \quad &\Rightarrow \quad \psi_\gamma=-\theta^{-1}u_\gamma, \quad \gamma=0,\dots,3\\
        v^\alpha \psi_5=-v^\alpha \theta^{-1}\pi\quad &\Rightarrow \quad \psi_5=-\theta^{-1}\pi\,.
      \end{split}
        \]

        We write $\psi^T=-\theta^{-1}(u_0,\vec u,g,\pi)$, where $g$ has to satisfy the identity

        \begin{equation}\label{eq:1}
        \begin{split}
         -\delta^\alpha_\mu s-\nu p\theta^{-1} u^\alpha u_\mu&= g \delta^\alpha_\mu
          +\theta^{-1}u_\beta \Bigg((\e+2\nu(p+\pi))u^\alpha u^\beta  u_\mu \\
          &\quad +(\e+\nu(p+\pi))(\delta_\mu^\alpha u^\beta+\delta_\mu^\beta u^\alpha)\Bigg)\\
          &\quad+\theta^{-1}\pi (\delta^\mu_\alpha (\C+\nu)+ \nu u^\alpha u_\mu)\,.
     \end{split}
   \end{equation}

      We calculate
      \[
        \begin{split}
 &       \theta^{-1}u_\beta \left((\e+2\nu(p+\pi))u^\alpha u^\beta  u_\mu +(\e+\nu(p+\pi))(\delta_\mu^\alpha u^\beta+\delta_\mu^\beta u^\alpha)\right)\\
&        =-\theta^{-1}\nu (p+\pi) u^\alpha u_\mu-\theta^{-1} (\e+\nu(p+\pi))\delta_\mu^\alpha
\end{split}     \]
and we see that \eqref{eq:1} holds true for  $\alpha=0,\dots,3$ provided that
  \begin{equation}\label{eq:8}
  g
  =\e-\theta s+\nu p-\C\pi\,,
\end{equation}
  which is the free enthalpy with natural variables $\theta, p,\pi$. 
  


Above we have made the assumption that $s$ is a strictly concave function of its arguments $\e,\nu,\C$. This is equivalent to the requirement that $g$ is a strictly concave function of $\theta,p,\pi$. In particular, this implies that $\theta,g,\pi$ provide coordinates for the thermodynamical state space (by that we mean that $(u^1,u^2,u^3,\theta,g,\pi)^T$ provide a set of coordinates on $\Omega$). Hence $\psi$ also provides a set of coordinates.
  
Following the general arguments from \cite{ruggeri1981main}, we now set
\[
  X^\alpha(\psi)=\psi^T F^\alpha(\psi) -S^\alpha(\psi)
\]
and will show in the upcoming theorem that this defines a Godunov-Boillat system under suitable assumptions on the state function. We use the following notation in the statement: Whenever $z_1,z_2,z_3$ provide coordinates for the thermodynamical state space, we write
\[
  \left.\frac{\partial h}{\partial z_1}\right|_{z_2,z_3}
\]
for the partial derivative with respect to $z_1$ of the function $(z_1,z_2,z_3)\mapsto h(z_1,z_2,z_3)$, and when $h_1,h_2,h_3$ are functions depending only on the thermodynamic state variables, we write
\[
  \frac{D(h_1,h_2,h_3)}{D(z_1,z_2,z_3)}
\]
for the Jacobain determinant $\det \left(\frac{\partial h_i}{\partial z_j}\right)_{i,j=1,2,3}$.
We also write
\[
  f=\e+\nu(p+\pi)\,.
  \]

\begin{theorem}
\label{thm:main}
  Suppose that $(\theta,p,\pi)\mapsto g(\theta,p,\pi)$ is strictly concave and that additionally
    \begin{equation}\label{eq:5}
    \begin{split}
      1+\frac{\nu^2}{f}\left.\frac{\partial p}{\partial \nu}\right|_{\theta,\C}&\geq 0\\
1+\frac{\nu^2}{f}\left(-\left.\frac{\partial\pi}{\partial\C}\right|_{\theta,p} +\left.\frac{\partial p}{\partial\nu}\right|_{\theta,\pi}+2\left.\frac{\partial\pi}{\partial\nu}\right|_{\theta,p}\right)&\geq 0
  \end{split}
\end{equation}
   on $\Omega$. 
  Then the function $X^\alpha(\psi)$ defines a Godunov-Boillat system.
\end{theorem}

\begin{proof}
  We need to verify \eqref{eq:3} and \eqref{eq:4}. Firstly,
  \[
    \frac{\partial X^\alpha}{\partial \psi_i}=(F^{\alpha})_i+\psi^T\frac{\partial F^\alpha}{\partial \psi_i}-\frac{\partial S^\alpha}{\partial \psi_i}\,.
\]  
Obviously
\[
  \psi^T\frac{\partial F^\alpha}{\partial \psi_i}-\frac{\partial S^\alpha}{\partial \psi_i}=\sum_{j=0}^5\left(\psi^T\frac{\partial F^\alpha}{\partial U_j}-\frac{\partial S^\alpha}{\partial U_j}\right)\frac{\partial U_j}{\partial \psi_i}=0\,,
\]
and hence \eqref{eq:3} is proved.

\bigskip

For an arbitrary future-directed 4-vector $\xi^\beta$ with $\xi_\beta \xi^\beta=-1$, we set
\[
  \begin{split}
  \bar X&=\xi_\alpha X^\alpha\\
  \bar F&=\xi_\alpha F^\alpha
\end{split}
\]

In order to verify \eqref{eq:4}, it suffices to prove
  \begin{equation}\label{eq:6}
    \begin{split}
      0&< \sum_{i,j=0}^5 v_iv_j\partial_{\psi^i}\partial_{\psi^j} \bar X(\psi)\\
& = \sum_{i,j=0}^5 v_iv_j\partial_{\psi_i}\bar F_j(\bar\psi)\\
   & = \sum_{i,j,k=0}^5 v_iv_j\partial_{\psi_j}\psi_k\partial_{\psi_i}\bar F_k(\bar\psi)\\
   &= \sum_{i,j,k=0}^5 \tilde v_i\tilde v_j \partial_{U_j}\psi_k\partial_{U_i}\bar F_k
  \end{split}
\end{equation}
for every $v\in T_P\Omega\equiv \R^6$, $P\in \Omega$, 
  where we have used the notation
\[\tilde v_j=\sum_{l=0}^5\frac{\partial U_j}{\partial  \psi_l}v_l\,.
\]
The vector $v\in T_P\Omega$ will be fixed from now on, and 
 for any (smooth) function $z:\Omega\to\R$ we will write
\[
  \delta z= \sum_{l=0}^5 \tilde v_l \partial_{U_l} z\,.
  \]




\newcommand{\Q}{\mathcal Q}

  We set
  \[
    \begin{split}
    \Q&:= \sum_{i,j,k=0}^5 \tilde v_i\tilde v_j \partial_{U_i}\psi_k\partial_{U_j}\bar F_k\\
    &= \sum_{k=0}^5  \delta\psi_k\delta\bar F_k
  \end{split}
  \]
  and write $\bar v=\xi_\alpha v^\alpha$, $\bar u=\xi_\alpha u^\alpha$. We have by $\psi=-\theta^{-1}(u_0,\vec u^T,g,\pi)^T$, 
    \[
      \theta^2\Q=\sum_{\beta=0}^4(u_\beta\delta\theta-\theta\delta u_\beta)\delta \bar F_\beta+(g\delta\theta-\theta\delta g)\delta \bar v+(\pi\delta \theta-\theta \delta\pi)\delta(\bar v(\C+\nu))
    \]
    In a lengthy but straightforward calculation, in which we use the identities \eqref{eq:8}, and 
    \[
      \begin{split}
      u^\beta\delta u_\beta&=0\\
      \delta g&=-s\delta\theta+\nu\delta p-\C\delta\pi\\
      \delta^2 g&=-\delta s\delta\theta+\delta \nu\delta p-\delta\C\delta\pi
    \end{split}
      \]
we obtain
    \begin{equation}\label{eq:9}
    \begin{split}
      \theta \Q&=\bar v(\delta^2 g-f\delta u_\beta\delta u^\beta )-2 \delta (p+\pi)\delta \bar u
  \end{split}
\end{equation} 

    We now introduce the notation $\delta \vec u=(\delta u^1,\delta u^2,\delta  u^3)^T$ and use the relations
    \[
      \begin{split}
      \delta u^0&=\sum_{i=1}^3 \frac{u_i\delta u^i}{u^0}\equiv\frac{\vec u\cdot \delta \vec u}{u^0}\\
      \delta u_\beta\delta u^\beta&=\sum_{i,j=1}^3\delta u^i(\delta_{ij}-(u^0)^{-2}u_iu_j)\delta u^j\\
      &= \delta\vec u^T\, (\id_{3\times 3}+\vec u\otimes \vec u)^{-1}\,\delta\vec u\\
      \delta \bar u&=\left(\vec\xi-\frac{\xi^0}{u^0}\vec u\right)\cdot \delta \vec u 
    \end{split}
\]
      and rewrite \eqref{eq:9} as 
  \begin{equation}\label{eq:10}
  \theta \Q= \bar v\delta^2 g- \bar v f \delta\vec u^T\, (\id_{3\times 3}+\vec u\otimes \vec u)^{-1}\,\delta\vec u-2 \delta(p+\pi)\delta \vec u\cdot \left(\vec\xi-\frac{\xi^0}{u^0}\vec u\right)\,.
\end{equation}
To ``complete the square'',  we add and subtract 
\[
  \begin{split}
  &\left(\vec\xi-\frac{\xi^0}{u^0}\vec u\right)^T\left(\id_{3\times 3}+\vec u\otimes \vec u\right)\left(\vec\xi-\frac{\xi^0}{u^0}\vec u\right)\frac{(\delta(p+\pi))^2}{\bar v f}\\
  &\qquad= (\xi_\beta\xi^\beta+|u_\beta\xi^\beta|^2)\frac{(\delta(p+\pi))^2}{\bar v f}\\
  &\qquad= (\bar u^2-1)\frac{(\delta(p+\pi))^2}{\bar v f}
\end{split}
\]
on the right hand side in \eqref{eq:10} and obtain
\[
  \theta \Q= \bar v\delta^2 g- \bar v f  \vec w^T(\id_{3\times 3}+\vec u\otimes \vec u)^{-1}\vec w+(\bar u^2-1)\frac{(\delta(p+\pi))^2}{\bar v f}\,,
\]
where
\[
  \vec w= \delta \vec u+(\id_{3\times 3}+\vec u\otimes \vec u)\left(\vec\xi-\frac{\xi^0}{u^0}\vec u\right)^T\frac{\delta(p+\pi)}{\bar v f}\,.
    \]
    Multiplying with $\nu^2 \bar v$ (which is negative since both $\xi^\mu$ and $v^\mu$ are future directed time-like), we obtain finally
      \begin{equation}\label{eq:11}
      \begin{split}
      \bar v  \nu^2 \theta \Q&=\bar u^2\delta^2 g-\bar u^2 f\vec w^T(\id_{3\times 3}+\vec u\otimes \vec u)^{-1}\vec w+(\bar u^2-1)\frac{\nu^2}{f} (\delta(p+\pi))^2\\
      &=-(\bar u^2-1)\left(-\delta g^2-\frac{\nu^2}{f} (\delta(p+\pi))^2\right)
      \underbrace{+\delta g^2-\bar u^2 f\vec w^T(\id_{3\times 3}+\vec u\otimes \vec u)^{-1}\vec w}_{<0\, \forall\, v\in T_p\Omega\setminus\{0\}}
    \end{split}
  \end{equation}
We have that $\bar u^2=|\xi_\beta u^\beta|\geq 1$ and hence the right hand side is negative for every $v\in T_P\Omega\setminus\{0\}$ if the matrix
    \begin{equation*}
\tilde G=\left(
  \begin{array}{ccc}
    -\partial_\theta^2g& -\partial_\theta\partial_pg&-\partial_\theta\partial_\pi g\\
    -\partial_\theta\partial_pg&-\partial_p^2 g-f^{-1}\nu^2&-\partial_p\partial_\pi g-f^{-1}\nu^2\\
    -\partial_\theta\partial_\pi g&-\partial_p\partial_\pi g-f^{-1}\nu^2&-\partial_\pi^2 g-f^{-1}\nu^2
  \end{array}\right)
\end{equation*}
is non-negative definite. By the strict concavity of $g$, we have $-\partial_\theta^2 g>0$. We claim that the two  conditions in \eqref{eq:5} are precisely the non-negativity of the other two leading principal minors of $\tilde G$, which proves negativity of the right hand side in \eqref{eq:11}. Indeed, 
\[
  \begin{split}
\det \left(
  \begin{array}{cc}
    -\partial_\theta^2g& -\partial_\theta\partial_pg\\
    -\partial_\theta\partial_pg&-\partial_p^2 g-f^{-1}\nu^2
  \end{array}\right)
&= \frac{D(-s,\nu,-\C)}{D(\theta,p,-\C)}\left(1+\frac{\nu^2}{f}\frac{D(\theta,p,-\C)}{D(-s,\nu,-\C)}\frac{D(-s,\nu,-\C)}{D(\theta,\nu,-\C)}\right)\\
&=\frac{D(-s,\nu,-\C)}{D(\theta,p,-\C)}
\left(  1+\frac{\nu^2}{f}\left.\frac{\partial p}{\partial \nu}\right|_{\theta,\C}\right)\geq 0
\end{split}
\]
   and
  \[
    \begin{split}
    \det \tilde G&=-\frac{D(-s,\nu,-\C)}{D(\theta,p,\pi)}-\frac{\nu^2}{f}\left(\frac{D(-s,\nu,-\C)}{D(\theta,p,-\C)}   +\frac{D(-s,\nu,-\C)}{D(\theta,\nu,\pi)}  -2 \frac{D(-s,\nu,-\C)}{D(\theta,\nu,p)}\right)\\
    &=-\frac{D(-s,\nu,-\C)}{D(\theta,p,\pi)}
    \left(1+\frac{\nu^2}{f}\left(\frac{D(\theta,p,\pi)}{D(\theta,p,-\C)}+\frac{D(\theta,p,\pi)}{D(\theta,\nu,\pi)}-2\frac{D(\theta,p,\pi)}{D(\theta,\nu,p)}\right)\right)\\
     &=-\det D^2_{(\theta,p,\pi)} g\left(1+\frac{\nu^2}{f}\left(-\left.\frac{\partial\pi}{\partial\C}\right|_{\theta,p} +\left.\frac{\partial p}{\partial\nu}\right|_{\theta,\pi}+2\left.\frac{\partial\pi}{\partial\nu}\right|_{\theta,p}\right)\right)\geq 0\,.
  \end{split}
    \]
This completes the proof.      
  \end{proof}

  In the introduction we have noted that the generalized MIS model with bulk viscosity allows for discontinuous shock solutions.  Lax's entropy condition guarantees that weak shocks will have positive entropy production across the shock, see e.g.~\cite{dafermos2005hyperbolic}. A proof of this fact in the relativistic setting (i.e., for causal covariant Godunov-Boillat systems) can be found in  \cite{ruggeri1981main}. We obtain the straightforward corollary:

\begin{corollary}
  \label{thm:shock}
  For weak  Lax shocks of (weak) solutions of the generalized MIS model with bulk viscosity \eqref{eq:12}, the entropy production is positive across the shock.
   \end{corollary}

    \bibliographystyle{alpha}
\bibliography{relhydro}

\newcommand{\etalchar}[1]{$^{#1}$}
\begin{thebibliography}{BDH{\etalchar{+}}21}

\bibitem[BDH{\etalchar{+}}21]{bemfica2021nonlinear}
F.~S. Bemfica, M.~M. Disconzi, V.~Hoang, J.~Noronha, and M.~Radosz.
\newblock Nonlinear constraints on relativistic fluids far from equilibrium.
\newblock {\em Phys. Rev. Lett.}, 126(22):222301, 2021.

\bibitem[BDN19]{bemfica2019causality}
F.~S. Bemfica, M.~M. Disconzi, and J.~Noronha.
\newblock Causality of the {E}instein-{I}srael-{S}tewart theory with bulk
  viscosity.
\newblock {\em Phys. Rev. Lett.}, 122(22):221602, 2019.

\bibitem[Daf05]{dafermos2005hyperbolic}
C.~M. Dafermos.
\newblock {\em Hyperbolic conservation laws in continuum physics}, volume~3.
\newblock Springer, 2005.

\bibitem[DHR20]{disconzi2020breakdown}
M.~M. Disconzi, V.~Hoang, and M.~Radosz.
\newblock Breakdown of smooth solutions to the {M\"u}ller-{I}srael-{S}tewart
  equations of relativistic viscous fluids.
\newblock {\em arXiv preprint arXiv:2008.03841}, 2020.

\bibitem[Dis19]{MR3927412}
M.~M. Disconzi.
\newblock On the existence of solutions and causality for relativistic viscous
  conformal fluids.
\newblock {\em Commun. Pure Appl. Anal.}, 18(4):1567--1599, 2019.

\bibitem[Fre19]{freistuhler2019relativistic}
H.~Freist{\"u}hler.
\newblock Relativistic barotropic fluids: A {G}odunov--{B}oillat formulation
  for their dynamics and a discussion of two special classes.
\newblock {\em Arch. Rat. Mech. Anal.}, 232(1):473--488, 2019.

\bibitem[Fre20]{freistuhler2020class}
H.~Freist{\"u}hler.
\newblock A class of {H}adamard well-posed five-field theories of dissipative
  relativistic fluid dynamics.
\newblock {\em J. Math. Phys.}, 61(3):033101, 2020.

\bibitem[Fri54]{friedrichs1954symmetric}
K.~O. Friedrichs.
\newblock Symmetric hyperbolic linear differential equations.
\newblock {\em Commun. Pure Appl. Math.}, 7(2):345--392, 1954.

\bibitem[FT14]{freistuhler2014causal}
H.~Freist{\"u}hler and B.~Temple.
\newblock Causal dissipation and shock profiles in the relativistic fluid
  dynamics of pure radiation.
\newblock {\em Proc. Roy. Soc. A}, 470(2166):20140055, 2014.

\bibitem[FT17]{freistuhler2017causal}
H.~Freist{\"u}hler and B.~Temple.
\newblock Causal dissipation for the relativistic dynamics of ideal gases.
\newblock {\em Proc. Roy. Soc. A}, 473(2201):20160729, 2017.

\bibitem[FT18]{freistuhler2018causal}
H.~Freist{\"u}hler and B.~Temple.
\newblock Causal dissipation in the relativistic dynamics of barotropic fluids.
\newblock {\em J. Math. Phys.}, 59(6):063101, 2018.

\bibitem[GL90]{geroch1990dissipative}
R.~Geroch and L.~Lindblom.
\newblock Dissipative relativistic fluid theories of divergence type.
\newblock {\em Phys. Rev. D}, 41(6):1855, 1990.

\bibitem[GL91]{geroch1991causal}
R.~Geroch and L.~Lindblom.
\newblock Causal theories of dissipative relativistic fluids.
\newblock {\em Annals of Physics}, 207(2):394--416, 1991.

\bibitem[God61]{godunov1961interesting}
S.~K. Godunov.
\newblock An interesting class of quasilinear systems.
\newblock In {\em Dokl. Acad. Nauk SSSR}, volume 139, pages 521--523, 1961.

\bibitem[IS79]{israel1979transient}
W.~Israel and J.~M. Stewart.
\newblock Transient relativistic thermodynamics and kinetic theory.
\newblock {\em Ann. Phys.}, 118(2):341--372, 1979.

\bibitem[JCVL96]{jou1996extended}
D.~Jou, J.~Casas-V{\'a}zquez, and G.~Lebon.
\newblock Extended irreversible thermodynamics.
\newblock {\em Extended Irreversible Thermodynamics}, pages 41--74, 1996.

\bibitem[Lax57]{lax1957hyperbolic}
P.~D. Lax.
\newblock Hyperbolic systems of conservation laws {II}.
\newblock {\em Commun. Pure Appl. Math.}, 10(4):537--566, 1957.

\bibitem[Lax73]{lax1973hyperbolic}
P.~D. Lax.
\newblock {\em Hyperbolic systems of conservation laws and the mathematical
  theory of shock waves}.
\newblock Society for Industrial and Applied Mathematics, 1973.

\bibitem[MR13]{muller2013rational}
I.~M{\"u}ller and T.~Ruggeri.
\newblock {\em Rational extended thermodynamics}, volume~37.
\newblock Springer Science \& Business Media, 2013.

\bibitem[M{\"u}l67]{muller1967paradoxon}
I.~M{\"u}ller.
\newblock Zum {P}aradoxon der {W\"a}rmeleitungstheorie.
\newblock {\em Zeitschrift f{\"u}r Physik}, 198(4):329--344, 1967.

\bibitem[{\"O}tt05]{ottinger2005beyond}
H.~C. {\"O}ttinger.
\newblock {\em Beyond equilibrium thermodynamics}.
\newblock John Wiley \& Sons, 2005.

\bibitem[RR19]{romatschke2019relativistic}
P.~Romatschke and U.~Romatschke.
\newblock {\em Relativistic fluid dynamics in and out of equilibrium: and
  applications to relativistic nuclear collisions}.
\newblock Cambridge University Press, 2019.

\bibitem[RS81]{ruggeri1981main}
T.~Ruggeri and A.~Strumia.
\newblock Main field and convex covariant density for quasi-linear hyperbolic
  systems: relativistic fluid dynamics.
\newblock {\em Annales de l'IHP Physique th{\'e}orique}, 34(1):65--84, 1981.

\bibitem[RS15]{MR3379901}
T.~Ruggeri and M.~Sugiyama.
\newblock {\em Rational extended thermodynamics beyond the monatomic gas}.
\newblock Springer, Cham, 2015.

\bibitem[Rug16]{MR3497743}
T.~Ruggeri.
\newblock Non-linear maximum entropy principle for a polyatomic gas subject to
  the dynamic pressure.
\newblock {\em Bull. Inst. Math. Acad. Sin. (N.S.)}, 11(1):1--22, 2016.

\bibitem[Tay11]{MR2744149}
M.~E. Taylor.
\newblock {\em Partial differential equations {III}. {N}onlinear equations},
  volume 117 of {\em Applied Mathematical Sciences}.
\newblock Springer, New York, second edition, 2011.

\bibitem[ZHYY15]{zhu2015conservation}
Y.~Zhu, L.~Hong, Z.~Yang, and W.-A. Yong.
\newblock Conservation-dissipation formalism of irreversible thermodynamics.
\newblock {\em Journal of Non-Equilibrium Thermodynamics}, 40(2):67--74, 2015.

\end{thebibliography}

\end{document}